\documentclass[12pt]{amsart}
\usepackage{amssymb}
\def\a{\alpha}
\def\b{\beta}
\def\d{\delta}
\def\s{\sigma}

\def\l{\lambda}
\def\L{\Lambda}
\def\g{\gamma}
\def\e{\epsilon}  
\def\oa{\mathop{\otimes}_A}
\def\ob{\mathop{\otimes}_B}

\theoremstyle{definition}

\newcommand{\Z}{\mathbb{Z}}

\newcommand{\ul}{\underline}

\newtheorem{theorem}{Theorem}[section]
\newtheorem{lemma}[theorem]{Lemma}
\newtheorem{example}[theorem]{Example}
\newtheorem{examples}[theorem]{Examples}
\newtheorem{definition}[theorem]{Definition}
\newtheorem{proposition}[theorem]{Proposition}
\newtheorem{remark}[theorem]{Remark}
\newtheorem{corollary}[theorem]{Corollary}
\newtheorem{problems}[theorem]{Problems}
\newtheorem{question}[theorem]{Question}

\begin{document}
\title{Peirce decompositions, idempotents and rings}

\author{P. N. \'Anh}
\address{R\'enyi Institute of Mathematics, Hungarian Academy of Sciences,
1364 Budapest, Pf.~127, Hungary}
\email{anh@renyi.hu}
\author{G. F. Birkenmeier}
\address{Department of Mathematics, University of Louisiana at Lafayette, Lafayette, LA 70504-1010, USA}
\email{gfb1127@louisiana.edu}
\author{L. van Wyk}
\address{Department of Mathematical Sciences, Stellenbosch University,
P/Bag X1, \hfill\break Matieland~7602, Stellenbosch, South Africa }
\email{LvW@sun.ac.za}
\thanks{Corresponding author: G. F. Birkenmeier   \\ \indent The first author was
supported by the National Research, Development and Innovation Office NKHIH K119934, FAPESP during his stay at the University of S\~ao Paulo, and by the Vietnamese Institute of Mathematics for his stay in Hanoi, Vietnam.  \\ \indent The third author was supported by the
National Research Foundation (of South Africa) under grant no.~UID 72375. Any
opinion, findings and conclusions or recommendations expressed in this
material are those of the authors and therefore the National Research
Foundation does not accept any liability in regard thereto.} 
\date{Draft \today}

\begin{abstract}
Idempotents dominate the structure theory of rings. The Peirce decomposition induced by an idempotent provides a natural environment for defining and classifying new types of rings. This point of view offers a way to unify and to expand the classical theory of semiperfect rings and idempotents to much larger classes of rings. Examples and applications are included.
\end{abstract}

\maketitle

\section {Introduction}
\label{notions}

Since the coordinatization of projective and continuous geometries (see \cite{vn}), it is well-known that idempotents induce direct sum decompositions of regular representations which determine a structure of rings, provided that the rings have enough idempotents. This wealth of idempotents can be ensured if rings are proper matrix rings, i.e., they are $n$-by-$n$ for $n > 1$. An idempotent $e=e^2$ in a ring  $R$ not necessarily with unity induces  the (two-sided) Peirce decomposition
$$R=eRe\oplus eR(1-e)\oplus (1-e)Re\oplus (1-e)R(1-e),$$
or more transparently, $e$ induces on $R$ the generalized matrix ring structure
$$R= \left[ \begin{array}{cc}
eRe & eR(1-e) \\
(1-e)Re & (1-e)R(1-e) \end{array}\right],$$
with the obvious matrix addition and multiplication. Here $eRe$ ($=\{ere \  \vert \ 
r\in R\}$), $eR(1-e), \ (1-e)Re$ and $(1-e)R(1-e)$ are abelian subgroups of $R$, where the abbreviated notation $eR(1-e)$ stands formally for the set $\{e(r-re)=er-ere \ \vert \ r\in R\}$; and similarly, $(1-e)Re = \{re-ere \ \vert \ r\in R\}$, $(1-e)R(1-e)=\{r-er-re+ere \ \vert \ r\in R\}$.  Henceforth, there are generally two ways to treat idempotents concerning their structural influence. The first is an internal way given by the classical Peirce decompositions; the second way is an external one provided by generalized (or formal) matrix  rings. It is well known (e.g., see \cite{abivw1}) that with each Peirce decomposition, we can associate a generalized matrix ring; and with each generalized matrix ring, we can associate a Peirce decomposition.  Observe that a Morita context is 
a $2$-by-$2$ generalized matrix ring. Recall that a Morita context is a quadruple $(A, B, _A\negthinspace M_B, _B\negthinspace N_A)$ of rings $A$ and $B$ and bimodules $_AM_B$ and 
$_BN_A$, together with $(A,A)$- and $(B,B)$-bimodule homomorphisms

$$(-,-):M\ob N\longrightarrow _A\negthinspace A_A, \,\,\,\,\,\,\,\,\,\,\,\,\,\,\, [-,-]:N\oa M\longrightarrow _B\negthinspace B_B,$$

\noindent satisfying the conditions 

$$(m,n)m_1=m[n,m_1]\qquad{\rm and}\qquad [n,m]n_1=n(m,n_1) $$

\noindent of associativity for all $m, m_1\in M, \ n, n_1\in N.$ It is not necessary to require $A$ and~$B$ to be unital rings and $M$ and $N$ to be unitary bimodules. Consequently, every Morita context provides a generalized matrix ring 
$$R = \left[ \begin{array}{cc}
A & M \\
N & B \end{array}\right],$$
endowed with the usual matrix addition and multiplication by using bimodule homomorphimsms $(-,-), [-,-]$;
and vice versa in the case when at least one of~$A$ and $B$ is unital, by putting $e= \left[ \begin{array}{cc}
1 & 0 \\
0 & 0 \end{array}\right]$ or $e= \left[ \begin{array}{cc}
0 & 0 \\
0 & 1 \end{array}\right]$, one obtains a Peirce decomposition.

Generalized matrix rings, in particular Morita contexts, provide an efficient way to obtain rings with prescribed idempotents of a certain type. Then, using the prescribed idempotents to obtain Peirce decompositions, one can obtain further information about the rings.  For example, a ring with unity is a 2-by-2 generalized upper triangular matrix ring if and only if it has a left semicentral 
 idempotent which is neither 0 nor 1. Moreover, the Peirce decomposition may provide a means to unify a class of generalized matrix rings. For example, renumbering pairwise orthogonal idempotents 
leads to formally different generalized matrix rings which can be transformed from one to another by appropriate interchanging of rows and columns, respectively. However, the
associated Peirce decomposition is the same, because addition is commutative. 

The associativity condition imposed on Morita contexts is satisfied trivially if the considered bilinear products are trivial, i.e., zero. This naturally suggests  the notion of Peirce idempotents. An idempotent $e=e^2\in R$ is called \emph{Peirce trivial} if $eR(1-e)Re=0=(1-e)ReR(1-e)$ (see~\cite{abivw1}). By defining the class of rings which are indecomposable relative to the Peirce trivial concept (i.e., rings in which 0 and~1 are the only Peirce trivial idempotents) one obtains building blocks for a new decomposition theory (see Definition \ref{peirce}).  We refer to Peirce's original paper \cite{pe} for decompositions induced by idempotents. Other aspects and related properties of matrix and generalized matrix rings can be found also in \cite{abivw1},  \cite{bipar}, \cite{bodvw}, \cite{krt}, 
\cite{mamesvw} and \cite{s}. 

In this paper we devote our attention to the investigation of $n$-Peirce rings. In contrast to our other work in \cite{abivw1}, in this article we give a coordinatization-free treatment, i.e., we look for results which are independent of particular generalized matrix ring representations. In Section 2, the main result shows that one can develop a structure theory of Peirce rings similar to that of Bass for semiperfect rings. Thorough discussions on conditions weakening Peirce trivial idempotents can be found in~\cite{abivw1}. In Section 3, following the program suggested by Jacobson's classic (see~\cite{j}), we define so-called trivial idempotents relative to certain radicals, like {\bf J}-trivial and {\bf B}-trivial idempotents; and we sketch the process of how to lift results on semisimple factors to the rings. This is closely related to the classical theory of lifting idempotents modulo radicals. Applications of our theory are developed in the last section. In particular, we show that a variety of well known and useful conditions produce an $n$-Peirce ring with a generalized matrix representation whose diagonal rings are 1-Peirce rings which satisfy the respective condition.  Moreover, we provide many well known classes of rings for which an $n$-Peirce ring has a generalized matrix representation in which each diagonal ring is 1-Peirce and in the respective class.

A word about notation and convention: in the rest of this paper all rings are unital and all modules are unitary. When a ring $R$ with an idempotent $e^2=e\in R$ is viewed as a generalized matrix ring 
$R = \left[ \begin{array}{cc}
eRe & eR(1-e) \\
(1-e)Re & (1-e)R(1-e) \end{array}\right]$,
then the identity element of the rings $A=eRe$ and $B=(1-e)R(1-e)$ is 1 by convention, not $e$ or $1-e$, respectively. This convention will simplify and make routine calculations transparent. We consider an element $r = 1 \cdot r \cdot 1 = [e+(1-e)]r[e+(1-e)]$ both as a sum $r=ere+er(1-e)+(1-e)re+(1-e)r(1-e)$ and as a formal matrix $r=\left[ \begin{array}{cc}
ere & er(1-e) \\
(1-e)re & (1-e)r(1-e) \end{array}\right]$.

\medskip

\section{General structure theory}
\label{basicth}

For the sake of self-containment we provide the following definition (see \cite{abivw1}).

\begin{definition}
\label{peirce} An idempotent $e=e^2$ in a ring $R$ is called \emph{inner Peirce trivial} if $eR(1-e)Re=0$. Dually, $e$ is \emph{outer Peirce trivial} if $1-e$ is an inner Peirce trivial. An idempotent $e$ is \emph{Peirce trivial} if it is both inner and outer Peirce trivial. A ring~$R$ is a \emph{$0$-Peirce ring} if it has only one element $1=0$, and $R$ is called a \emph{Peirce ring}, or more precisely, a $1$\emph{-Peirce ring} if 0 and 1 (with $1\ne0$) are the only Peirce trivial idempotents of $R$. Inductively, for a natural number $n>1$, a ring~$R$ is called an \emph{$n$-Peirce ring}
if there is a Peirce trivial idempotent $e^2=e\in R$ such that $eRe$ is an $m$-Peirce ring for some 
$m, \ 1\leq m<n$, and $(1-e)R(1-e)$ is an $(n-m)$-Peirce ring. An idempotent 
$e\in R$ is called an 
\emph{n-Peirce idempotent} if $eRe$ is an $n$-Peirce ring. In particular,  $e=e^2\in R$ is called a \emph{1-Peirce idempotent} if $eRe$ is a $1$-Peirce ring.
\end{definition}

\begin{remark}
\label{recentral} Since all central idempotents are Peirce trivial, every 1-Peirce ring is indecomposable as a ring.  In particular, if a ring $R$ is semiprime or Abelian then both inner and outer Peirce trivial idempotents in $R$ are central; and such a ring $R$ is 1-Peirce if and only if $R$ is indecomposable as a ring. Recall that a ring is {\it Abelian} if its idempotents are central. Peirce trivial idempotents generalize the notion of semicentral idempotents which occur naturally in the structure of 2-by-2 generalized triangular matrix rings. For a natural number $n$, $n$-Peirce rings are generalizations of $n$-strongly triangular matrix rings (see \cite{avw2}), or in another terminology, rings with a complete set of triangulating idempotents (see~\cite{bihkipa}). For a thorough and subtle analysis of inner and outer Peirce idempotents, see~\cite{abivw1}. It is also worth noting that for an idempotent $e^2=e\in R$ the set $e+eR(1-e)$ is characterized in \cite[Part~II, Chapter~II, Lemma~2.7]{vn}  as the set of idempotents $f^2=f\in R$ such that $e$ and $f$ generate the same right ideal. 
\end{remark}

The following characterization (which is related to \cite[Corollary 3.6]{abivw1}) of Peirce trivial idempotents is obvious in view of Definition \ref{peirce}.

\begin{proposition}
\label{obvi1} Let $e=e^2\in R$ and $I=eR(1-e)+(1-e)Re$. Then $e$ is Peirce trivial if and only if $I$ is an ideal of $R$. 
\end{proposition}

Direct matrix computations (see \cite{avw1} and \cite{avw2}) yield the following:

\begin{proposition}
\label{basic} Let $e=e^2\in R$ be a Peirce trivial idempotent, and put $A=eRe, \ B=(1-e)R(1-e), \ M=eR(1-e)$ and $N=(1-e)Re$. For arbitrary elements $m\in M$ and $n\in N$ the element $f = \left[ \begin{array}{cc}
1 & m \\
n & 0 \end{array}\right]$ is an idempotent, the rings $A$ and $ fRf=\left\{\left[ \begin{array}{cc} a & am \\
na & 0 \end{array}\right]: a\in A \right\}$ are isomorphic under the map
$\varphi$, sending
$a\in A$ to $\varphi(a) = \left[ \begin{array}{cc}
a & am \\
na & 0 \end{array}\right],$ and $B$ and 
$(1-f)R(1-f)=\left\{\left[ \begin{array}{cc} 0 & -mb \\
-bn & 1 \end{array}\right] : b\in B \right\}$ are isomorphic under the map
 $\varrho$, sending
$b\in B$ to $\varrho(b) = \left[ \begin{array}{cc}
0 & -mb \\
-bn & b \end{array}\right]$. Also, the modules $_RRe$ and $_RRf$ are isomorphic by sending $e\mapsto ef$ and $f\mapsto fe$. 
Moreover, $M=fR(1-f), \ N=(1-f)Rf$
and the identity maps on $M$ and $N$ are $(\varphi, \varrho)$-bimodule isomorphisms, i.e., for any $a\in eRe, \ b\in (1-e)R(1-e), \ x\in M$ and $y\in N$ one has 
$axb=\varphi(a)x\varrho(b)$ and $bya=\varrho(b)y\varphi(a)$. Consequently, $f$ is also a Peirce trivial idempotent in $R$.
\end{proposition}

Simple formal calculations with matrices also show the following result.

\begin{lemma}
\label{calculate} If $e^2=e\in R$ and $g^2=g\in eRe$ are Peirce trivial idempotents in $R$ and $eRe$, respectively, then for any $m\in eR(1-e)$ and $n\in (1-e)Re$: \begin{enumerate}
\item the element 
$h = \left[ \begin{array}{cc}
g & gm \\
ng & 0 \end{array}\right]$ is an inner Peirce trivial idempotent in $R$; 
\item the rings
$gRg$ and $hRh$ are isomorphic; 
\item the modules $_RRg$ and $_RRh$ are isomorphic.
\end{enumerate}
\end{lemma}

\begin{remark}
\label{warning}  It can seen from \cite[Example 3.9]{abivw1} that if $e$ is a Peirce trivial idempotent in $R$ and g is a Peirce trivial idempotent in $eRe$, then $g$ need not be Peirce trivial in $R$; but $g$ is inner Peirce trivial in $R$ (see \cite[Lemma 3.8(1)]{abivw1}). 
Observe that in general a product of two Peirce trivial idempotents is not even an idempotent.  Let $R$ be the 
     2-by-2 upper triangular matrix ring over a ring $A$.  Then $\left[ \begin{array}{cc} 1 & 0  \\ 0 & 0 \end{array}\right] 
		                                                             \left[ \begin{array}{cc} 0 & 1  \\ 0 & 1 \end{array}\right]$ 
is a product of Peirce trivial idempotents which is not an idempotent. 
\end{remark}

The following result provides basic properties of Peirce trivial idempotents.

\begin{proposition}
\label{ortho} Let $e\in R$ be a Peirce trivial idempotent in a ring $R$. Then any idempotent $f\in R=\left[ \begin{array}{cc}
A & M \\
N & B \end{array}\right],$ where $A=eRe,\, B=(1-e)R(1-e), \, M=eR(1-e)$ and $N=(1-e)Re$, can be written as a sum of two orthogonal idempotents $\a = \left[ \begin{array}{cc}
g & gm \\
ng & 0 \end{array}\right]$ and $\b = \left[ \begin{array}{cc}
0 & mh \\
hn & h \end{array}\right] \ $ ($f=\a+\b$ and ${\a}{\b}={\b}{\a}=0$) for appropriate $g^2=g\in A$,\, $h^2=h\in B,\, m\in M$ and $n\in N$. Furthermore,
\begin{enumerate}
\item $\a$ and $\b$ are Peirce trivial idempotents in the ring $fRf$; 
\item the modules $_RRf$ and $_RRf_e$, where $f_e=g+h$, are isomorphic;
\item $f$ is a Peirce trivial
idempotent of $R$ if and only if $f_e$ is a Peirce trivial idempotent of $R$.
\end{enumerate}
 Moreover, if $f$ is a Peirce trivial idempotent of $R$, then $g$ and $h$ are also Peirce trivial idempotents of $A$ and $B$, respectively. They are inner Peirce trivial idempotents of $R$, but not necessarily outer Peirce trivial idempotents of $R$. The same is true for both $\a$ and $\b$, i.e., they are inner Peirce trivial idempotents of~$R$.
\end{proposition}

\begin{proof} Since $f$ can be written uniquely as the generalized matrix $f = \left[ \begin{array}{cc}
g & m \\
n & h \end{array}\right]$ for uniquely determined elements $g\in A, \ h\in B, \ m\in M$ and $n\in N$, the equality
$f^2=f$ implies that 
$$g^2=g, \ \ h^2=h, \ \ m=gm+mh \ \ {\rm and} \ \ n=ng+hn,$$ 
which in turn implies that 
$$gmh=0 \ \ {\rm and} \ \ hng=0.$$
Let
$$\a = \left[ \begin{array}{cc}
g & gm \\
ng & 0 \end{array}\right] \ \ {\rm and} \ \ \b = \left[ \begin{array}{cc}
0 & mh \\
hn & h \end{array}\right].$$
Then one can verify directly that $\a={\a}^2, \ \b={\b}^2, \ f=\a+\b$ and ${\a}{\b}={\b}{\a}=0.$ 

To see that $\a$ and $\b$ are Peirce trivial idempotents of $fRf$, one has to verify that ${\a}fRf{\b}fRf{\a}={\a}R{\b}R{\a}=0={\b}fRf{\a}fRf{\b}={\b}R{\a}R\b$, which is obvious by observing the inclusions ${\a}R\b\subseteq M$ and ${\b}R\a\subseteq N.$ The modules $_RRf$ and~$_RRf_e$ are isomorphic by the equalities $f=ff_ef$ and $f_e=f_eff_e.$ Since $f=f_e+(gm+hn)+(mh+ng), \ gm+hn\in f_eR(1-f_e), \ mh+ng\in (1-f_e)Rf_e, \ f_e=f-(gm+hn)-(mh+ng), \ gm+hn\in fR(1-f)$ and $mh+ng\in (1-f)Rf$, it follows immediately in view of Proposition \ref{basic} that $f$ is Peirce trivial if and only if $f_e$ is such.

Assume now in addition that $f$ is a Peirce trivial idempotent of $R$. The idempotent $e$ is now the identity $1_A$ of $A$, i.e., $e=\left[ \begin{array}{cc}
1 & 0 \\
0 & 0 \end{array}\right],$ and similarly $1-e= \left[ \begin{array}{cc}
0 & 0\\
0 & 1 \end{array}\right]$ is the identity $1_B$ of $B$. The equality
$0=fR(1-f)Rf=(\a+\b)R(1-f)R(\a+\b)$ implies that $0={\a}eRe(1-f)eRef=gA(e-g)Afe=aA(e-g)Ag$, whence $g$ is an inner Peirce trivial idempotent of $A=eRe$.
On the other hand, the equality $0=(1-f)RfR(1-f)$ shows thst
$0=(1-f)eRefeRe(1-f)=(1-f)eAgAe(1-f)$, from which $0=e(1-f)eAgAe(1-f)e=
(e-g)AgA(e-g)$ follows. Therefore, $g$ is also an outer Peirce trivial idempotent of~$A$. Consequently, $g$ is indeed a Peirce trivial idempotent of $A$. By symmetry, $h$ is also a Peirce trivial idempotent of 
$B=(1-e)R(1-e)$.
The remaining assertions are now simply consequences of~Lemma~\ref{calculate}. 
\end{proof}

As an obvious consequence of Proposition \ref{ortho} and Remark \ref{warning}, routine matrix multiplication shows that: 

\begin{corollary}
\label{product} In the notation of Proposition \ref{ortho}, the products $efe$ and $(1-e)f(1-e)$ of a Peirce trivial idempotent 
$e$ of $R$ and an idempotent $f\in R$ are the idempotents $g$ and $h$ of $R$, respectively. Moreover, in the case of a Peirce trivial idempotent $f$, the idempotents $g$ and $h$ are inner Peirce trivial, but not necessarily outer Peirce trivial idempotents of $R$.
\end{corollary}

To justify Definition \ref{peirce}, one has to show that $n$ is an invariant of an $n$-Peirce ring, i.e, $n$ does not depend on the choice of Peirce trivial idempotents in $R$. This fact is shown in the following result.

\begin{theorem}
\label{main} Let $R$ be an $n$-Peirce ring and $f^2=f\in R$ be an arbitrary Peirce trivial idempotent. Then $fRf$ is a $k$-Peirce ring for some $k\leq n$, and $(1-f)R(1-f)$  is  an  $(n-k)$-Peirce ring.
\end{theorem}

\begin{proof}
We use induction. The case $n=1$ is obvious from Definition \ref{peirce}. Assume now that $n>1$ and that the theorem is true for all $m<n$. Consider  an $n$-Peirce ring~$R$ defined by a Peirce trivial idempotent $e^2=e\in R$ such that $eRe$ is an $m$-Peirce ring $(1\leq m<n)$ and $(1-e)R(1-e)$ is an $(n-m)$-Peirce ring. For simplifying calculations put $A=eRe, \ M=eR(1-e), \ N=(1-e)Re$ and $B=(1-e)R(1-e)$, and write the elements of $R$ as generalized matrices $r = \left[ \begin{array}{cc}
a & m \\
n & b \end{array}\right]$. Therefore, if $f^2=f\in R$ is an arbitrary Peirce trivial idempotent in $R$, then in view of Proposition \ref{ortho}, for the unique generalized matrix representation $f = \left[ \begin{array}{cc}
g & m \\
n & h \end{array}\right],$ with uniquely determined elements $g\in A, \ h\in B, \ m\in M$ and $n\in N,$ by putting
$\a = \left[ \begin{array}{cc}
g & gm \\
ng & 0 \end{array}\right]$ and $\b = \left[ \begin{array}{cc}
0 & mh \\
hn & h \end{array}\right]$, one has that $f=\a+\b, \ {\a}{\b}={\b}{\a}=0,$ $\a$ and $\b$ are Peirce trivial idempotent of $fRf$, and $g$ and $h$ are Peirce trivial idempotents of $A$ and $B$, respectively. Without loss of generality, we may assume that $f\neq 0, 1.$

By the induction hypothesis applied to both $A$ and $B$, the rings $gAg$ and $hBh$ are $p$- and $q$-Peirce rings for some $p, \ 0\leq p\leq m$, and some $q, \ 0\leq q\leq n-m$, respectively, such that at least one of the two inequalities is proper by the extra assumption on~$f$. For the sake of simplicity, by putting $t=gm=gt, \ u=ng=ug, \ v=mh=vh$ and $w=hn=hw$, one has

$$\a = \left[ \begin{array}{cc}
g & t \\
u & 0 \end{array}\right] \quad {\rm and} \quad \b = \left[ \begin{array}{cc}
0 & v \\
w & g \end{array}\right].$$
Routine matrix calculations show that 
$${\a}fRf{\a}={\a}R{\a} = \left\{\left[ \begin{array}{cc}
gag & gat \\
uag & 0 \end{array}\right]\colon \  gag\in gAg\right\}$$
and
$$ (f-\a)fRf(f-\a)={\b}R{\b} = \left\{\left[ \begin{array}{cc}
0 & vbh \\
hbw & hbh \end{array}\right]\colon \ hbh\in hBh\right\},$$ 
whence ${\a}R{\a}$ and ${\b}R{\b}$ are isomorphic to $gAg$ and $hBh$ via the maps

$$\left[ \begin{array}{cc}
gag & gat \\
uag & 0 \end{array}\right] \longmapsto gag\in gAg$$
and
$$\left[ \begin{array}{cc}
0 & vbh \\
hbw & hbh \end{array}\right]\longmapsto hbh\in hBh,$$ 
respectively. Consequently $fRf$ is an $(p+q)$-Peirce ring.  Moreover, the left $R$-modules $_RRg$ and $_RR\a$ are isomorphic, as are $_RRh$ and $_RR\b$. By the same manner and by the induction hypothesis we have also that $(1-f)R(1-f)$ is an $(n-p-q)$-Peirce ring, completing the proof.
\end{proof}

Since an idempotent is either Peirce trivial or not Peirce trivial, Theorem \ref{main} suggests the following dichotomy.

\begin{definition}
\label{peircedim}
A ring $R$ has {\it Peirce dimension} 0 if it has only one element~$1=0$, and $R$ has {\it Peirce dimension} $n \ (n>0)$ if it is an $n$-Peirce ring. All other rings are said to have {\it infinite Peirce dimension}.
\end{definition}

As an obvious consequence of Definition \ref{peircedim} and Theorem \ref{main} we have:

\begin{corollary}
\label{additive} The Peirce dimension is additive, i.e., the Peirce dimension of a finite direct sum of rings is the sum of the Peirce dimensions of the direct summands. In particular, if $e\in R$ is an arbitrary Peirce trivial idempotent, then the Peirce dimension of $R$ is the sum of the Peirce dimensions of $eRe$ and
$(1-e)R(1-e)$.
\end{corollary}

The following consequence deals with rings of infinite Peirce dimension.

\begin{corollary}
\label{infinitedim} A ring $R$ has infinite Peirce dimension if and only if there is an infinite set of nonzero pairwise distinct idempotents 
$e_0=1, e_1, \ldots, e_n, \ldots $ such that $e_{i+1}$ is a $1$-Peirce idempotent of $e_iRe_i$ for every $i=0, 1, 2, \ldots \ $.
\end{corollary}

The class of $1$-Peirce rings covers rings with a variety of different properties. It contains, for example, all prime rings and rings with only the two idempotents $1$ and $0$. Furthermore, all matrix rings over local rings are also $1$-Peirce rings, as is easily verified. The problem of characterizing or describing the class of $1$-Peirce rings seems to be quite interesting. Since the class of $n$-Peirce rings is even larger, one needs additional invariants to get a closer look at them. 

The following invariant is an obvious consequence of Definition \ref{peirce}.

\begin{definition}
\label{dyadic} Let $I$ be a finite nonempty set. A {\it partition of} $I$
is a finite set of nonempty pairwise disjoint subsets whose union is $I$. A {\it dyadic partition} of $I$ is a partition into two disjoint subsets. A partition $\l$ is called a {\it dyadic refinement} of a partition $\g$ if all elements of $\g$, with one exception, are elements of~$\l$, and the exceptional element is a union of two elements of $\l$. A set
$\Lambda=\bigl\{\l_0=\{I\}, \l_1, \l_2, \cdots, \l_k\bigr\}$ of partitions
$\l_i$ is called a \emph{complete dyadic set of partitions} if $\l_{i+1}$ is a dyadic refinement of $\l_i$ for all $i=0,\ldots, k-1$ and all elements of $\l_k$ are sets comprising only one element. Therefore $k=n-1$ if $I$ has $n$ elements. 
\end{definition}

For a subset $I$ of $\{1, 2, \ldots, n\}$ and a set $\{e_1,e_2,\ldots,e_n\}$ of idempotents in a ring~$R$ the sum  $\sum\limits_{i\in I}e_i$ is denoted by~$e_I$.

The following important characterization of $n$-Peirce rings is an easy consequence of Definitions \ref{peirce} and \ref{dyadic}.

\begin{corollary}
\label{structure} A ring $R$ is an $n$-Peirce ring if and only if there are $n$ pairwise orthogonal $1$-Peirce idempotents $e_1, \ldots, e_n$ whose sum is 1, and a complete dyadic 
set~$\L=\biggl\{\l_0=\bigl\{\{1, 2, \ldots, n\}\bigr\},  \l_1, \l_2, \cdots, \l_k\biggr\}$ of partitions of $\{1, 2, \ldots, n\}$ such that for an exceptional element $I$ of $\l_i, \ i=0,\ldots\,k-1,$ which is a union of two elements $J$ and $L$ of $\l_{i+1}, \ e_J$ is a Peirce trivial idempotent of the subring $e_IRe_I$. 
\end{corollary}

\begin{proof} The sufficiency is obvious. For the necessity
we use induction on $n$. The claim is obvious for $n=2, 3$. Let $n>3$ and assume that the claim is true for all $m<n$.  By Definition \ref{peirce} there is a Peirce trivial idempotent $E\in R$ such that $ERE$ and $FRF$ (with $F=1-E$) are $n_1$- and $n_2$-Peirce rings, respectively, with $n_1+n_2=n, \  n_1n_2\neq 0$.
Therefore by the induction hypothesis there are pairwise orthogonal 1-Peirce idempotents $e_1, \ldots, e_{n_1}$ in $ERE$  and $f_1, \ldots, f_{n_2}$ in $FRF$ together with complete dyadic sets 
$\Lambda_1=\{\a_0, \ldots, \a_{n_1-1}\}$ and $\Lambda_2=\{\d_0, \ldots, \d_{n_2-1}\}$ of partitions of $\{1,\ldots,n_1\}$ and $\{1,\ldots,n_2\},$ making $ERE$ and $FRF$ $n_1$- and $n_2$-Peirce rings, respectively. Now, putting $e_i=f_{i-n_1}$ for all $i, \ n_i<i<n+1$, we obtain a set of $n$ pairwise orthogonal idempotents $\{e_1,\ldots, e_n\}$ with sum 1. Partitions of $\{1,\ldots,n_2\}$ in the set $\Lambda_2$ define partitions of the set $\{n_1+1,\ldots,n\}$ by sending $i, \ 0<i<n_2+1$, to $n_1+i$, whence the set $\Lambda_2$ of partitions of $\{1,\ldots,n_2\}$ defines the set 
$\Lambda_3=\{\b_0, \ldots, \b_{n_2-1}\}$ of partitions of the set $\{n_1+1, \ldots, n\}$. Putting $\l_0=\{1, \ldots, n\}$ and $\l_{i+1}=\a_i\cup\b_0$ for all $i, \ 0\leq i<n_1$, and $\l_{n_1+i-1}=\a_{n_1-1}\cup \b_i$ for all $i, \ 0<i<n_1$, one obtains a required complete dyadic set $\Lambda=\{\l_0, \ldots, \l_{n-1}\}$ of partitions of $\{1,\ldots,n\}$, which completes the proof. 
\end{proof}

With the notation of Corollary \ref{structure}, for an $n$-Peirce ring $R$ together with $n$ pairwise orthogonal 1-Peirce idempotents $e_i$ with sum 1 and a complete dyadic set $\L$ of partitions, the subset ${\frak D}(R)^-=\sum\limits_{i\neq j}e_iRe_j$ (see \cite{abivw1}) is a nilpotent ideal of nilpotency index at most $n$. 

This simple assertion is based on the next observation. If $I$ is an exceptional subset in the partition ${\l}_{k-1}$ which is a disjoint union of two subsets $J, K$ in $\l_k$, put 
${\frak D}_{\l_i}(R)^-=e_JRe_K+e_KRe_J$. Then in view of Proposition \ref{obvi1}, ${\frak D}_{\l_i}(R)^-$ is an ideal of
$E_IRE_I$ with square 0. Since ${\frak D}(R)^-=\sum\limits_{\l_i}{\frak D}_{\l_i}(R)^-$, the claim is proved. 

We will see later that ${\frak D}(R)^-$ is independent of the choice of a set of idempotents~$e_i$. 

It is worth stating separately a result which is similar to the classical Wedderburn
Principal Theorem:

\begin{corollary}
\label{structure1} Under the above notation, an $n$-Peirce ring $R$ is a direct sum of~${\frak D}(R)^-$ and a subring which is a direct sum of $n$ 1-Peirce rings.
\end{corollary}

The converse of this simple result is, in general, not true. It is quite interesting to find sufficient conditions such that a ring $R$ is an $n$-Peirce ring if it has a direct decomposition $R=S\oplus D$ of a subring $S$, which is a direct product of $n$ 1-Peirce rings, and a nilpotent ideal $D$ of nilpotency index (at most) $n$.
  
The proof of Corollary \ref{structure} shows that there are several complete dyadic sets of partitions that define the same $n$-Peirce ring within a given set of pairwise orthogonal 1-Peirce idempotents whose sum is 1,  and there is freedom and room in numbering the idempotents under consideration. 

Now we give an obvious way to construct one such set of idempotents with a possible numbering (indexing) and a complete dyadic set of partitions. Since Definition \ref{peirce} is deductive, first the identity $E_0=1$ is an orthogonal sum of two proper Peirce trivial idempotents $E_0=E_{00}+E_{01}$ of $R$. Second,  each of $E_{00}$ and $E_{01}$ is again an orthogonal sum of two proper Peirce trivial idempotents in the associated rings $E_{00}RE_{00}$ and $E_{01}RE_{01}$, respectively, except the case when they are 1-Peirce idempotents. In this exceptional case, they are elements of a required set of pairwise orthogonal 1-Peirce idempotents. 

Continuing in this way, after finitely many steps one obtains, for an $n$-Peirce ring~$R$, a sequence $e_1,\ldots,\, e_n$ of pairwise  orthogonal 1-Peirce idempotents with sum~1 and a complete dyadic set of partitions $\l_0=\{I\}\subseteq \l_1\subseteq \l_2\subseteq \cdots \subseteq \l_k$ such that for an exceptional element $I$ of $\l_i, \ i=0,\ldots\,k-1$, which is the union of two elements $J$ and $L$ of $\l_{i+1}, \ e_J$ is a Peirce trivial idempotent of the subring~$e_IRe_I$. Furthermore, one can index them such that for each index~$i<n$ there is an index~$j_i, \ i<j_i\leq n$, maximal with respect to the property that $e_i$ is a Peirce trivial idempotent in the ring $E_iRE_i$, where $E_i=e_i+e_{i+1}+\cdots +e_{j_i}$. 

However, a sequence $e_1,\ldots,\, e_n$ of pairwise  orthogonal 1-Peirce idempotents with sum 1 such that for each index $i<n$ there is an index $j_i, \ i<j_i\leq n$, maximal with respect to the property that $e_i$ is a Peirce trivial idempotent in the ring $E_iRE_i$, where $E_i=e_i+e_{i+1}+\cdots +e_{j_i}$, is 
not sufficient to ensure that a ring is an $n$-Peirce ring. The reason is that such a sequence is far from ensuring that there exists a subsum of the $e_i$'s which is Peirce trivial in the ring. 

Furthermore, if $f^2=f\in R$ is an arbitrary 1-Peirce idempotent of $R$, then according to Proposition \ref{ortho} together with its notation, in the expression $f=\a+\b$ of $f$ as a sum of two orthogonal Peirce trivial idempotents $\a$ and $\b$ of $fRf,$ one of~$\a$ and~$\b$ must be 0, say, $\b=0$. Then $g$ is a 1-Peirce idempotent in a subring~$eRe$ which is an $m$-Peirce ring with $m<n$.  Therefore, after finitely many steps one finds
an idempotent $e_i$, uniquely determined by $f$, such that there is a 1-Peirce idempotent $g\in e_iRe_i$, with $f = \left[ \begin{array}{cc}
g & gm \\
ng & 0 \end{array}\right]$ or $f = \left[ \begin{array}{cc}
0 & mg \\
gn & g \end{array}\right]$, where $m$ and~$n$ are appropriate elements 
of $e_iR(1-e_i)$ and $(1-e_i)Re_i$, respectively. Note that~$g$ is, in general, not equal to $e_i$, the identity of the ring $e_iRe_i$, as one can see easily in the case of a matrix ring over a division ring. 

These arguments lead to:

\begin{proposition}
\label{peirce3} Under the hypothesis and notation of Corollary \ref{structure}, any 1-Peirce idempotent $f$ in an $n$-Peirce ring $R$
determines uniquely an $i \in\{1, 2, \ldots, n\}$ and a 1-Peirce idempotent $g\in e_iRe_i$ such that 
\begin{enumerate}
\item $_RRf$ and $_RRg$ are isomorphic; 
\item $f = \left[ \begin{array}{cc}
g & gm \\
ng & 0 \end{array}\right]$ or $f = \left[ \begin{array}{cc}
0 & mg \\
gn & g \end{array}\right]$ for appropriate $m\in e_iR(1-e_i)$ and $n\in (1-e_i)Re_i$. If $f$ is a Peirce trivial idempotent, then $g=e_i$.
\end{enumerate}
\end{proposition}

\begin{remark}
\label{warning2} If $f$ is an $m$-Peirce idempotent in an $n$-Peirce ring $R$, then $f$ is an orthogonal sum of $m$ pairwise orthogonal 1-Peirce idempotents $f_j$, and in view of Proposition \ref{peirce3}, there are uniquely determined indices $i_j$ associated to $j$ and a 1-Peirce idempotent $g_j\in  e_{i_j}Re_{i_j}$ such that 
\begin{enumerate}
\item $_RRf_j$ and $_RRg_j$ are isomorphic, and 
\item $f_j = \left[ \begin{array}{cc}
g_j & g_jm_j \\
n_jg_j & 0 \end{array}\right]$ or $f_j = \left[ \begin{array}{cc}
0 & m_jg_j \\
g_jn_j & g_j \end{array}\right]$ for appropriate $m_j\in e_{i_j}R(1-e_{i_j})$ and $n_j\in (1-e_{i_j})Re_{i_j}$. 
\end{enumerate}
However, it is possible that the indices $i_j$ are the same for different indices $j$ as in the following example. 
Let
$$R = \left[ \begin{array}{ccc} \Z/8\Z & 4\Z/8\Z & 2\Z/8\Z\\
2\Z/8\Z & \Z/8\Z & 2\Z/8\Z \\
2\Z/8\Z & 2\Z/8\Z & \Z/8\Z \end{array}\right].$$
Then one can check that $R$ is a 1-Peirce ring and $fRf$ is a
2-Peirce ring, where $f=f_1+f_2$ and
$$f = \left[ \begin{array}{ccc} 1 & 0 & 0\\
0 & 1 & 0 \\
0 & 0 & 0 \end{array}\right], \,\, f_1 = \left[ \begin{array}{ccc} 1 & 0 & 0\\
0 & 0 & 0 \\
0 & 0 & 0 \end{array}\right]\,\,{\rm and} \,\, f_2 = \left[ \begin{array}{ccc} 0 & 0 & 0\\
0 & 1 & 0 \\
0 & 0 & 0 \end{array}\right],$$
whence the corresponding indices $i_1$ and $i_2$ coincide. Although $R$ seems to be quite simple, it has $2^{20}$ elements! One can verify that $R$ is a $1$-Peirce ring in the same way as 
in \cite[Example 5.9]{abivw1}. 

Note that \cite[Example 5.9]{abivw1} also provides  an example of a $1$-Peirce ring with 3 pairwise orthogonal $1$-Peirce idempotents whose sum is 1. There is another handy short way to check this assertion without computation, as follows. Observing that~$R$ is a semiperfect ring, in fact, a finite ring, all complete sets of pairwise orthogonal primitive idempotents of $R$ are conjugate, i.e., any such set can be transformed into another one by an inner automorphism (see Lemma \ref{inner} below). Hence it suffices to check for the complete set $\{f_1, f_2, 1-(f_1+f_2)\}$ of pairwise orthogonal primitive idempotents of $R$, which is obvious. In this way one can construct quite a  large class of
semiperfect 1-Peirce rings. 

This example shows also that an $n$-Peirce ring can contain a proper $l$-Peirce idempotent $e$, i.e., $e\neq 1, 0$, with $l>n$. For example, let $A=\Z/2^n\Z \ (n>2)$ and let $X=2A, \ Y=2^{n-1}A$. The above method implies immediately that the finite generalized $n\times n$ matrix ring
$$R=\begin{pmatrix} A&Y&Y&\dotsb&Y&X\\Y&A&Y&\dotsb&Y&X^2\\Y&Y&A&\dotsb&Y&X^3\\ \vdots&\vdots&\vdots&\ddots&\vdots\vdots\\
Y&Y&Y&\dotsb&A&X^{n-1}\\
X^{n-1}&X^{n-2}&X^{n-3}&\dotsb&X&A\end{pmatrix}$$ 
is a $1$-Peirce ring together with an idempotent $f^2=f\in R$ such $fRf$ is
an $(n-1)$-Peirce ring. Consequently, a $1$-Peirce ring $R$ can contain an idempotent $f$ such that the subring $fRf$ has an arbitary finite Peirce dimension. 
\end{remark} 

In order to describe a relation between two sets of pairwise orthogonal 1-Peirce idempotents showing that a ring is an $n$-Peirce ring, let us recall 
the following more general, but folklore, result.

\begin{lemma}
\label{inner} If $\{e_1, \ldots , e_n\}$ and $\{f_1, \ldots, f_n\}$ are two sets of pairwise orthogonal idempotents in a ring $R$ whose sums are 1,  such that the modules $_RRe_i$ and $_RRf_i$ are isomorphic for all
$i=1, 2, \ldots , n$, then there is an invertible element $s\in R$ such that $se_is^{-1}=f_i$ for all $i=1, 2,\ldots, n$.
\end{lemma}

\begin{proof} By assumption, for each $i=1, 2,\ldots , n$ there are elements $s_i, t_i\in R$ such that
the equalities $e_is_if_i=s_i, \ f_it_ie_i=t_i, \ s_it_i=e_i$ and $t_is_i=f_i$ hold. Put $s=s_1+\cdots +s_n$ and $t=t_1+\cdots + t_n$. Simple calculations show that $st=ts=1$ and $sf_it=sf_is^{-1}=e_i$ for all $i=1, 2, \ldots, n$.
\end{proof}

In spite of Proposition \ref{peirce3}, we are now in a position to give some partial positive results showing some similarity to the theory of semiperfect rings.

\begin{theorem}
\label{peirce4} Let $R$ be an $n$-Peirce ring defined by $n$ pairwise orthogonal 1-Peirce idempotents $e_1, \ldots, e_n$ with sum 1 and a complete dyadic set of partitions 
$\bigl\{\l_0=\{I\}, \l_1, \l_2, \ldots,\l_k\bigr\}$ of $\{1,2,\ldots,n\}$ such that for an exceptional element $I$ of $\l_i, \ i=0,\ldots,k-1,$ which is the union of two elements $J$ and $L$ of $\l_{i+1}$, $e_J$ is a Peirce trivial idempotent of the subring $e_IRe_I$. 
\begin{enumerate}
\item If $\{f_1, \ldots, f_m\}$, with $m\le n$, is any set  of pairwise orthogonal 1-Peirce idempotents with sum 1, then $m=n$ and there is an invertible element $s\in R$ and a permutation $\s$ of the set $\{1, \ldots , n\}$ such that 
$f_{{\s}(i)}=se_is^{-1}$ for all $i=1, \ldots ,n$. 
\item If $f^2=f\in R$ is a $k$-Peirce idempotent, then for each index $j$ in a representation of $f=\sum\limits_{j=1}^{k}f_j$ as a sum of~$k$ pairwise orthogonal 1-Peirce idempotents, there exists  an index $i$ and a $1$-Peirce idempotent $\e_{i_j}\in e_iRe_i$ such that $_RRf_j$ is isomorphic to 
$_RR\e_{i_j}$, where the $\e_{i_j}$'s are pairwise orthogonal appropriate $1$-Peirce idempotents contained in  $e_iRe_i$. Consequently, $_RRf$ is isomorphic to $_RR\e, \, \e=\sum\limits_{j=1}^{k}\e_{i_j}$. Moreover, if $f$ is a Peirce trivial idempotent, then $f$ is a $k$-Peirce idempotent
for some $k\leq n$, and in the above representation of $f=\sum\limits_{j=1}^{k}f_j$ as a sum of $k$ pairwise orthogonal 1-Peirce idempotents $f_j$, for each
index $j$ one has $\e_{i_j}=e_i$, i.e., the correspondence $j\mapsto i_j$ is injective.  
\end{enumerate}
\end{theorem}

\begin{proof} (1) By Proposition \ref{peirce3}, for each $f_j$ there is a uniquely determined $e_{i_j}$ with a 1-Peirce idempotent 
$g_{i_j}\in e_{i_j}Re_{i_j}$ such that $f_j$ and $g_{i_j}$ are equal modulo ${\frak D}(R)^-$. Since the factor of $R$ by ${\frak D}(R)^-$ is a direct product of $n$ rings $e_iRe_i$, and $\sum f_j$ maps to~1 in this factor ring, together with $m\leq n$, one gets that all the $g_{i_j}$ are distinct and each $g_{i_j}$ is the identity $e_{i_j}$ of the ring $e_{i_j}Re_{i_j}$. This shows that $m=n$, and that $_RRf_j$ and $_RRe_{i_j}$ are isomorphic $R$-modules, whence the existence of an inner automorphism of $R$ sending the $e_i$ onto the $f_i$ follows in view of Lemma \ref{inner}. 

(2) We use the notation in the proof of Theorem \ref{main}. 
Let $e$ be a Peirce trivial idempotent, ensuring that $R$ is an $n$-Peirce ring, and put $A=eRe, \ M=eR(1-e), \ N=(1-e)Re$ and $B=(1-e)R(1-e)$, where $A$ is an $m$-Peirce ring and~$B$ is an $(n-m)$-Peirce ring for some $m, \ 1\leq m < n$. For a $k$-Peirce idempotent $f = \left[ \begin{array}{cc}
g & m \\
n & h \end{array}\right]$ with uniquely determined elements $g\in A, \ h\in B, \ m\in M$ and $n\in N,$ let
$\a = \left[ \begin{array}{cc}
g & gm \\
ng & 0 \end{array}\right]$ and $\b = \left[ \begin{array}{cc}
0 & mh \\
hn & h \end{array}\right]$. One has that $f=\a+\b, \ {\a}{\b}={\b}{\a}=0,$ $\ \a$ and $\b$ are Peirce trivial idempotents of $fRf$, whence they are again $l_1$- and $l_2$-Peirce idempotents with $l_1, l_2\leq k$ of $R$, respectively, in view of
Theorem~\ref{main} and Proposition \ref{ortho}. Simple formal matrix calculation shows that $gRg$ and $hRh$ are isomorphic to ${\a}R\a$ and ${\b}R\b$, respectively. Consequently, $g$ and $h$ are $l_1$- and $l_2$-Peirce idemptents of $R$, respectively. Then the obvious induction finishes the proof of the first part of (2). 

If $f$ is a Peirce trivial idempotent of $R$, then again by Theorem~\ref{main} $f$ is a $k$-Peirce idempotent for some $k\leq n$. In a representation of $f$ as a sum $\sum\limits_{j=1}^{k}f_j$ of $k$ pairwise orthogonal $1$-Peirce idempotents above, Proposition \ref{ortho} shows that, for each index $j$, the corresponding idempotent $\e_{i_j}$ associated with $f_j$ is a Peirce trivial idempotent of $e_iRe_i$, whence $\e_{i_j}=e_i$, as required. It is worth noting that, in view of Remark \ref{warning2}, $k > n$ can happen for Peirce idempotents $f$ which are not Peirce trivial. 
\end{proof}

One can see assertion (1) of Theorem \ref{peirce4} by using \cite[Theorem~5.7(2)]{abivw1}. By this result, $R$ is a $k$-Peirce ring for some $k, \ k\leq m\leq n$, whence $k=m=n$ by~Theorem~\ref{main}. 

As a consequence of the above proof we obtain immediately that $$\sum_{i\neq j}f_iRf_j\subseteq \sum_{i\neq j}e_iRe_j.$$ Since the role of $e_i$ and of $f_j$ are now quite symmetric, in view of Proposition \ref{ortho} and Theorem \ref{peirce4}, by interchanging the role of $f_i$ and of $e_i$ in the above inclusion, we obtain the equality 
$$\sum_{i\neq j}f_iRf_j=\sum_{i\neq j}e_iRe_j,$$
showing that:

\begin{corollary}
\label{equality} The ideal ${\frak D}(R)^-$ of an $n$-Peirce ring $R$ is independent of the choice of the set $\{e_1,e_2,\dots,e_n\}$ of $n$ pairwise orthogonal 1-Peirce idempotents with sum~1. In particular, a ring-theoretical direct sum  
$\sum\limits_i e_iRe_i$ of $n$ 1-Peirce subrings~$e_iRe_i$ is uniquely determined by~$R$ up to isomorphisms, i.e., independent of the choice of the corresponding $n$ pairwise orthogonal 1-Peirce idempotents whose sum is 1. Consequently, the  subrings $e_iRe_i, \ i=1, 2, \ldots, n,$ are also invariants of $R$ and the 
bimodules $_{e_iRe_i} {e_iRe_j} _{e_jRe_j}$ are uniquely determined up to bimodule isomorphisms, too.
\end{corollary}

Unfortunately, the converse of this result is not true. However, in view of \cite[Theorem 5.7(2)]{abivw1}  every ring with a complete set
$\{e_1, e_2, \dots, e_n\}$ of pairwise orthogonal 1-Peirce idempotents $e_i$ is a $k$-Peirce ring for some $k\leq n$, whence $R$ admits a Wedderburn-like principal decomposition described in Corollaries 
\ref{structure1}
and \ref{equality} with another complete set of $k$ pairwise orthogonal
1-Peirce idempotents. Moreover, \cite[Theorem 5.7(2)]{abivw1} together with the above consequences provides a very handy tool to determine if certain rings are 1-Peirce rings. To state the criterion, we define an auxiliary
notion. A subset $X$ of a ring $R$ is said to be \emph{nilpotent of index} $n$ if its \emph{ring closure}, i.e., the smallest additive group of $R$ containing $X$ and closed under multipilcation is a nilpotent ring (without identity) of index $n$.

\begin{corollary}
\label{test} Let $R$ be a ring with a complete set $\{e_1, e_2, \dots, e_n\}$ of pairwise orthogonal 1-Peirce idempotents $e_i$, i.e.,
$\sum e_i=1, e_ie_j={\delta}_{ij}e_i$. Then $R$ is a $1$-Peirce ring if for every subset $I\subseteq \{1, 2, \dots, n\}$ of at least two elements,
the nilpotency index of ${{\frak D}_I}^-=\sum\limits_{i, j\in I, \ i\ne j} e_iRe_j$ is bigger than the cardinality of $I$.
\end{corollary}

Note that ${{\frak D}_I}^-$ may not be nilpotent. Furthermore, it should be noted separately that \cite[Theorem 5.7(2)]{abivw1} is an efficient tool for constructing certain $n$-Peirce rings with prescribed properties.
In view of Corollary \ref{structure} one can refine the definition of $n$-Peirce rings by including a complete dyadic set of partitions of~$\{1,2,\ldots,n\}$ as an additional invariant.

\begin{definition}
\label{peirce1} A ring $R$ is called an $n$-Peirce ring associated with a complete dyadic set of partitions $\l_0=\{I\}\subseteq \l_1\subseteq \l_2\subseteq \cdots \subseteq \l_k$ of $\{1,2,\ldots,n\}$ if  
there are $n$ pairwise orthogonal $1$-Peirce idempotents $e_1, \ldots, e_n$ with sum 1 such that for an exceptional element $I$ of $\l_i, \ i=0,\ldots,k-1,$ which is the union of two elements~$J$ and~$L$ of $\l_{i+1}$, $ \ e_J$ is a Peirce trivial idempotent of the subring $e_IRe_I$. 
\end{definition}

It is worth noting that an $n$-Peirce ring in the sense of Definition \ref{peirce} can admit different complete dyadic sets of partitions of $\{1,2,\ldots,n\}$.
It is quite an  interesting combinatorial question to determine all complete dyadic set of partitions of $\{1,2,\ldots,n\}$ for an $n$-Peirce ring. This freedom would provide room for a combinatorial description of certain automorphisms of $n$-Peirce rings. 

To justify Definition \ref{peirce1} we give an example of a 4-Peirce ring associated with the complete dyadic set $\Bigl\{\l_0=\bigl\{\{1, 2, 3, 4\}\bigr\}, \l_1=\bigl\{\{1, 2\},\{3, 4\}\bigr\},  \l_2=\bigl\{\{1\}, \{2\}, \hfill\break\{3, 4\}\bigr\}, \l_3=\bigl\{\{1\}, \{2\}, \{3\}, \{4\}\bigr\}\Bigr\}$ of partitions of $\{1,2,3,4\}$ which does not contain a Peirce trivial 3-Peirce idempotent. Consider the field $K=\Z_2$ of 2 elements, together with the trivial bilinear forms $[-,-]=(-,-):\Z_2\otimes_{\Z_2} \Z_2\rightarrow \Z_2$, and let $A=B=\left[ \begin{array}{cc}
K & K \\
K & K \end{array}\right]$ be the generalized matrix ring induced by these
trivial bilinear forms. Let $M=N=A=B$, considering $_AM_B$ and $_BN_A$ as bimodules equipped with the trivial bilinear form $(-,-)_B: M\otimes_B N\rightarrow A, [-,-]_A: N\otimes_A M\rightarrow B$. Now the generalized matrix ring $R = \left[ \begin{array}{cc}
A & M \\
N & B \end{array}\right]$ induced by these bilinear forms is the required
example, as is easily verified by using the method described in Remark \ref{warning2}.

Since a complete dyadic set associated with a 2-Peirce or a 3-Peirce ring is unique, or equivalently,
a Peirce trivial idempotent that defines a 2-Peirce or 3-Peirce ring, can be chosen to be a 1-Peirce idempotent, we can make the definition of a complete dyadic set essentially simpler as follows.

\begin{definition}
\label{dyadic1} A set
$\Lambda=\bigl\{\l_0=\{I\}, \l_1, \l_2, \cdots, \l_k\bigr\}$ of partitions
$\l_i$ of a finite nonempty set $I$ is called a \emph{reduced dyadic set of partitions} if $\l_{i+1}$ is a dyadic refinement of $\l_i$ for all $i=0,\ldots, k-1$ and all elements of $\l_k$ are sets having at most three elements. 
\end{definition}

Definition \ref{dyadic1} simplifies Corollary \ref{structure} as
\begin{corollary}
\label{structure0} A ring $R$ is an $n$-Peirce ring if and only if there are $n$ pairwise orthogonal $1$-Peirce idempotents $e_1, \ldots, e_n$ with sum 1 and a reduced dyadic set~$\L=\biggl\{\l_0=\bigl\{\{1, 2, \ldots, n\}\bigr\},  \l_1, \l_2, \cdots, \l_k\biggr\}$ of partitions of $\{1, 2, \ldots, n\}$ such that for an exceptional element $I$ of $\l_i, \ i=0,\ldots,k-1,$ which is the union of two elements $J$ and $L$ of $\l_{i+1},  \ e_J$ is a Peirce trivial idempotent of the subring $e_IRe_I$ and for all  subsets $J$ in $\l_k$ the ring $e_JRe_J$ is a $\vert J\vert$-Peirce ring  where $\vert J\vert\in \{1, 2, 3\}$ denotes the cardinality of $J$. 
\end{corollary}

In the last part of this section we describe automorphisms of certain
$n$-Peirce rings associated with a complete dyadic set of partitions, generalizing the notion of strongly generalized triangular matrix rings. First we need the following definition.
 
\begin{definition}
\label{peirce2} An idempotent $e=e^2$ in a ring $R$ is called a \emph{strict 1-Peirce} idempotent if it is Peirce trivial and $eRe$ is a $1$-Peirce ring. Also, $e$ is called a {\it strict n-Peirce} idempotent if $e$ is Peirce trivial and $eRe$ is an $n$-Peirce ring. A ring $R$ is called  inductively a {\it strict $n$-Peirce} ring if there is a strict 1-Peirce $e\in R$ such that $(1-e)R(1-e)$ is a strict $(n-1)$-Peirce ring. {\it Strict 1-Peirce} rings are precisely $1$-Peirce rings, whence strict 2- and strict 3-Peirce rings coincide also with 2- and 3-Peirce rings, respectively.
\end{definition}

Strict $n$-Peirce rings are precisely $n$-Peirce rings associated with a complete dyadic set $\{\l_1,\ldots,\l_n\}$ of partitions of $\{1,2,\ldots,n\}$ given by $\l_k=\bigl\{\{1\}, \{2\},\ldots, \hfill\break \{k-1\}, I_k\bigr\}, \ k=1,\ldots, n,$ where $I_k=\{k, \ldots, n\}$.
Therefore strict $n$-Peirce rings are natural extensions of strongly generalized triangular matrix rings (see \cite{avw2}), or in alternative terminology, rings with a complete set of triangulating idempotents (see~\cite{bihkipa}).

\begin{remark}
\label{primitive} It is worth emphasizing the subtle difference between $1$-Peirce idempotents and strict $1$-Peirce idempotents. The former are not necessarily Peirce trivial while the latter are such idempotents. For example, all proper idempotents in a matrix rings over a division ring are $1$-Peirce idempotents, but they are never strict $1$-Peirce idempotents!
\end{remark}

In order to obtain a description of  isomorphisms between strict $n$-Peirce rings, one needs some technical preparation. 

Let $A$ be a strict $m$-Peirce ring defined by an ordered sequence $e_1, \ldots, e_m$ of pairwise
orthogonal Peirce idempotents with sum~1 such that every $e_i, \ i=1,\ldots,m-1,$ is Peirce trivial in the subring $A_i=(e_i+\cdots +e_m)A(e_i+\cdots +e_m).$ Then $A_1=A$. Letting $R_i=e_iAe_i$, we have $A_m=R_m$.  

Next, let $B$ be another strict $n$-Peirce ring defined by an ordered sequence $f_1, \ldots, f_n$ of pairwise orthogonal Peirce idempotents 
with sum~1 such that every $f_i, \ i=1,\ldots,n-1,$ is Peirce trivial in the subring $B_i=(f_i+\cdots +f_n)B(f_i+\cdots +f_n).$ Then $B_1=B$. Letting $S_i=f_iBf_i$, we have $B_n=S_n$. 

If $\sigma$ is any permutation of $\{1, \ldots, n\}$, put $f^\sigma_1:=f_{\sigma(1)}, \ldots, f^\sigma_n:=f_{\sigma (n)}$. According to this notation, if we write $g_i = f^\sigma_i$, then one can identify the above convention as follows: 
$$ S^{\sigma}_i:=S_{\sigma(i)}  = g_iBg_i,   B^\sigma_i:=B_{\sigma (i)} = (g_i+ \cdots + g_n)B(g_i + \cdots + g_n). $$

We are now in a position to describe isomorphisms between strict $n$-Peirce rings (see [3, Theorem]).

\begin{theorem}
\label{iso}
Let $A$ and $B$ be strict $m$- and strict $n$-Peirce rings defined by ordered sequences $e_1, \ldots, e_m$ and
$f_1, \ldots, f_n$ of pairwise orthogonal $1$-Peirce idempotents (with sum 1 in both cases) associated with the complete dyadic sets $\{\l_1,\l_2,\ldots,\l_m\}$ and $\{\l'_1,\l'_2,\ldots,\l'_n\}$ of partitions of $\{1,2,\ldots,m\}$ and $\{1,2,\ldots,n\}$, respectively, where $\l_k=\{\{1\}, \{2\},.., \{k-1\}, I_k\}, \ k=1,\ldots,m,$ and $\l'_{k'}=\{\{1\}, \{2\},.., \{k'-1\}, I_{k'}\}, \ k'=1,\ldots,n$, with $I_k=\{k, \ldots, m\}$ and $I_{k'}=\{k', \ldots, n\}$. Then $A$ and~$B$ are isomorphic via an isomorphism $\varphi: A \rightarrow B$ iff $m=n$ and there is a permutation
$\sigma$ of $\{1, \ldots, m\}$ together with
ring isomorphisms $\rho_i: R_i=e_iAe_i \rightarrow S^\sigma_i = 
S_{\sigma (i)}=f^\sigma_iBf^\sigma_i=f_{\sigma(i)}Bf_{\sigma(i)}, \ i=1,\ldots, m=n,$ and for $i = 1, \ldots, m-1$ there are elements 
$m_i \in M^\sigma_i=f^\sigma_iB^\sigma_i(1-f^\sigma_i)$ and $n_i\in  N^\sigma_i=(1-f^\sigma_i)B^\sigma_if^\sigma_i, $ 
ring isomorphisms $\varphi_{i+1}: A_{i+1} \rightarrow
B^\sigma_{i+1}, \ R_i-A_{i+1}$-bimodule isomorphisms
$\chi_i : e_iA_i(1-e_i) \rightarrow M^\sigma_i$ and $A_{i+1}-R_i$-bimodule
isomorphisms
$\delta_i :(1-e_i)A_ie_i \rightarrow N^\sigma_i,$
with respect to $\rho_i$ and $\varphi_{i+1}$, such that 
for $i =1, \ldots, m-1$ and 
$a_i = \left[ \begin{array}{cc} r_i & x_i \\
y_i & a_{i+1} \end{array} \right] \in A_i,$

$$\varphi_i(a_i) = \left[ \begin{array}{cc}
\rho_i(r_i) & \rho_i(r_i)m_i+ \chi_i (x_i) - m_i \varphi_{i+1}(a_{i+1})\\
n_i\rho_i(r_i)+\delta_i(y_i)-n_i\varphi_{i+1}(a_{i+1}) & \varphi_{i+1}(a_{i+1}) \end{array}\right].$$

\noindent Moreover, all isomorphisms between isomorphic rings $A$ and $B$ can
be described in this manner. (Keep in mind that $\varphi_1 = \varphi, \ 
\varphi_m = \rho_m; \ A_m  = R_m.$) In particular, the automorphism group of a strict $n$-Peirce ring can be inductively described in terms of ones of $1$-Peirce subrings and of related bimodules.
\end{theorem}

\begin{proof} Assume that $A$ and $B$ are isomorphic via $\varphi$. Then $\varphi(e_1)$ is a strict 1-Peirce idempotent in $B$. Therefore, as in the proof of Theorem \ref{main}, for the unique generalized matrix representation $\varphi(e_1) = \left[ \begin{array}{cc}
s_1 & m \\
n & b_2 \end{array}\right]$, with uniquely determined elements $s_1\in S_1, b_2\in B_2, m\in f_1B(1-f_1)$ and $n\in (1-f_1)Bf_1,$ by putting
$\a = \left[ \begin{array}{cc}
s_1 & s_1m \\
ns_1 & 0 \end{array}\right]$ and $\b = \left[ \begin{array}{cc}
0 & mb_2 \\
b_2n & b_2 \end{array}\right]$, one has that $\varphi(e_1)=\a+\b, \ {\a}{\b}={\b}{\a}=0,$ and $\a$ and $\b$ are Peirce trivial idempotents of $\varphi(e_1)B\varphi(e_1)$ as well as  of $S_1$ and~$B_2$, respectively.
Since $\varphi(e_1)B\varphi(e_1)$ is a $1$-Peirce ring, one of $\a$ and $\b$ must be 
0. If $\b=0$, then $\varphi(e_1)=\a$, and hence in this case one has $s_1=f_1$, the identity element of~$S_1$, and one puts $\sigma(1)=1$. If $\a=0$, then $\varphi(e_1)=\b$, and $b_2$ is a strict 1-Peirce idempotent of both $B_1=S$ and $B_2$. 

In this situation, one can repeat the process. Therefore after finitely many steps 
there exists a natural number $j=\sigma(1)$ such that $f_j$ is a strict 1-Peirce idempotent of $S$, for each $k<j$ there are elements $x_k\in f_kSf_j$  and $v_k\in f_jSf_k$, and for each $k>j$ there are elements $u_k\in f_jSf_k$ and $y_k\in f_kSf_j$ such that 

$$\varphi(e_1)=\left[ \begin{array}{ccccccccccc}
0&0&\cdots&\cdots&0&x_1&0&\cdots&\cdots&\cdots&0\\
0&0&0&\cdots&0&x_2&\vdots&\ddots&&&\vdots\\
\vdots&\ddots&\ddots&\ddots&\vdots&\vdots&\vdots&&\ddots&&\vdots\\
\vdots&&\ddots&\ddots&0&\vdots&\vdots&&&\ddots&\vdots\\
0&\cdots&\cdots&0&0&x_{j-1}&0&\cdots&\cdots&\cdots&0\\
v_1&v_2&\cdots&\cdots&v_{j-1}&1&u_{j+1}&u_{j+2}&\cdots&\cdots&u_n\\
0&\cdots&\cdots&\cdots&0&y_{j+1}&0&0&\cdots&\cdots&0\\
\vdots&\ddots&&&\vdots&y_{j+2}&0&0&0&\cdots&0\\
\vdots&&\ddots&&\vdots&\vdots&\vdots&\ddots&\ddots&\ddots&\vdots\\
\vdots&&&\ddots&\vdots&\vdots&\vdots&&\ddots&\ddots&0\\
0&\cdots&\cdots&\cdots&0&y_n&0&\cdots&\cdots&0&0\\
 \end{array}\right],$$

\noindent or equivalently, for $$m_1=x_1+\cdots +x_{j-1}+u_{j+1}+\cdots +u_n\in M^\sigma_1=M_{\sigma(1)}=f_jB(1-f_j)$$ and $$n_1=v_1+\cdots+ v_{j-1}+y_{j+1}+\cdots + y_n\in N^\sigma_1=N_{\sigma(1)}=
(1-f_j)Bf_j,$$  
$$\varphi(e_1)=\left[ \begin{array}{cc}
1 & m_1 \\
n_1 & 0 \end{array}\right] \in \left[ \begin{array}{cc}
f_jBf_j & f_jB(1-f_j) \\
(1-f_j)Bf_j & (1-f_j)B(1-f_j) \end{array}\right].$$
Therefore $\varphi$ induces  ring isomorphisms $\rho_1:e_1Ae_1=R_1\rightarrow 
\varphi(e_1)B\varphi(e_1)\cong B^\sigma_1=f_jBf_j=S_j=B^\sigma_1$ and $\varphi_2:A_2=(1-e_1)A(1-e_1)\rightarrow (1-\varphi(e_1))B(1-\varphi(e_1))\cong (1-f_j)B(1-f_j)=B^\sigma_2$, and the restrictions of $\varphi$ to $e_1A(1-e_1)$ and $(1-e_1)Ae_1$ define the bimodule isomorphisms $\chi_i$ and $\delta_i$ to $M^\sigma_1$ and $N^\sigma_1$, respectively. Consequently, if
$$a=a_1 = \left[ \begin{array}{cc} r_1 & x_1 \\
y_1 & a_2 \end{array} \right] \in A=A_1 = \left[ \begin{array}{cc} e_1Ae_1 & e_1A(1-e_1) \\
(1-e_1)Ae_1 & (1-e_1)A(1-e_1) \end{array} \right],$$ then 

$$\varphi(a) = \varphi_1(a_1)=\varphi(r_1)+\varphi(x_1)+\varphi(y_1)+\varphi(a_2)=$$

$$=\left[ \begin{array}{cc}
\rho_1(r_1)  & \rho_1(r_1)m_1+ \chi_1 (x_1) - m_1 \varphi_2(a_2)\\
n_1\rho_1(r_1)+\delta_1(y_1)-n_1\varphi_2(a_2) &  \varphi_2(a_2) \end{array}\right].$$
The theorem follows now easily by reduction.
\end{proof}
\smallskip

Observing that the proof of Theorem \ref{iso} also shows, for each index $i$, that  
$$\varphi(e_i)=\left[ \begin{array}{cc}
1 & m_i \\
n_i & 0 \end{array}\right] \in \left[ \begin{array}{cc}
S_{\sigma(i)} & f_{\sigma(i)}B(f_{\sigma(i+1)}+\cdots+f_{\sigma(n)}) \\
(f_{\sigma(i+1)}+\cdots+f_{\sigma(n)})Bf_{\sigma(i)} & B^{i+1}_{\sigma} \end{array}\right],$$ 
one obtains the following result in view of Lemma \ref{inner}.

\begin{theorem}
\label{innerauto} Let $R$ be a strict $n$-Peirce ring defined by two sequences
\newline $e_1, \ldots, e_n$ and $f_1, \ldots, f_n$ of pairwise orthogonal $1$-Peirce idempotents with sum 1 in each case. Then there is a permutation
$\sigma$ of $\{1,2,\ldots,n\}$ and a unit $s\in R$ such that $se_is^{-1}=f_{\sigma(i)}$ for all $i=1, \ldots, n$. 
\end{theorem}

Observe that not every permutation can occur in the description of the isomorphisms between and automorphisms of $n$-Peirce rings. Of course, such permutations form a subgroup of the symmetric group and this subgroup leaves invariant the class of all complete dyadic sets of partitions defining $R$. 

We conclude this section with some remarks on the automorphism group of an $n$-Peirce ring $R$ together with a complete set $I=\{e_1, e_2, \dots, e_n\}$ of pairwise orthogonal
1-Peirce idempotents. Any automorphism $\phi$ of $R$ transforms $I$ to the complete set $I_{\phi}=\{\phi(e_1), \phi(e_2), \dots, \phi(e_n)\}$ of pairwise orthogonal 1-Peirce idempotents. Therefore by Theorem \ref{peirce4} $\phi$ determines uniquely the permutation $\sigma_{\phi}$ and a unit $s_{\phi}$ such that $\phi(e_i)=s^{-1}_{\phi}e_{\sigma_{\phi}(i)}s_{\phi}$ for every~$i$. However, $s_{\phi}$ is not uniquely determined by $\phi$. To simplify notation, we write $\sigma$ and $s$ for $\sigma_{\phi}$ and $s_{\phi}$, respectively. These permutations $\sigma$ form a subgroup $\Omega_R$ of the symmetric group leaving invariant the class of all complete dyadic sets of partitions. It is clear that $\phi$ induces the automorphisms $\rho_i$ between subrings $R_i=e_iRe_i$ and $R^{\sigma}_i=e_{\sigma(i)}Re_{\sigma(i)}$ and the $(\rho_i,\rho_j)$-bimodules isomorphisms $\chi_{ij}$ between bimodules $_{R_i}e_iRe_{R_j}=R_{ij}$ and
$R^{\sigma}_{ij}=e_{\sigma(i)}Re_{\sigma(j)}$. One can describe automorphisms of $R$ in terms of permutations from $\Omega_R$ and isomorphisms
$\rho_i, \chi_{ij}$ and units $s$ in the way similar to one given in Theorem \ref{iso} by using a complete dyadic set of partitions defining $R$.

\section{Lifting process}
\smallskip

If $e^2=e\in R$ is a Peirce trivial idempotent, then 
$(ReR(1-e)R)^2=0=(R(1-e)ReR)^2$, whence $R$
is a direct product of the rings  $eRe$ and $(1-e)R(1-e)$, provided that $R$ is semiprime. This observation implies the following result, which is basic in the lifting process concerning the structure of rings. Recall from \cite{gw}, $R/\rho(R)$ is a semiprime ring for any supernilpotent radical $\rho$. The collection of supernilpotent radicals includes the prime, nil, Levitzki, Jacobson, and Brown-McCoy radicals.

\begin{theorem}
\label{observation} A semiprime $n$-Peirce ring is a direct product of $n$ semiprime $1$-Peirce rings (i.e., of $n$ semiprime indecomposable rings). In particular, an $n$-Peirce ring which is 
$\rho$-semisimple, for a supernilpotent radical $\rho$ (see \cite{gw}), is a direct product of~$n$ $\rho$-semisimple $1$-Peirce rings.
\end{theorem}

Since prime rings are clearly $1$-Peirce rings, it is quite natural to ask: How large is the class of semiprime $1$-Peirce rings? The following simple example indicates that the class of semiprime $1$-Peirce rings is quite extensive.

\begin{example}
Let $R=K[x,y,z]/I$ be the factor ring of the commutative polynomial ring in three variables $x, y, z$ by the ideal $I$ generated by the monomial $xyz$. Then $R$ is a semiprime $1$-Peirce ring which is not prime, because $R$ has only the two trivial idempotents 0 and 1. The ring $R=\{\frac{m}{n}: \ m,n\in \Z, \ (n,2)=(n,3)=(n,5)=1\}$ is a semilocal prime domain with nonzero Jacobson radical, namely the ideal generated by 30. Let $E$ be the minimal injective cogenerator of~$R$, i.e., $E$ is a direct sum of three quasi-cyclic abelian groups $C(2^{\infty}), C(3^{\infty})$ and $C(5^{\infty})$. Then the trivial extension of $R$ by $E$ is also a ring having only the two trivial idempotents, and with nonzero nil radical. Consequently, this ring is a non-prime $1$-Peirce ring. More generally, if $R$ is a ring with only the two trivial idempotents, e.g., a polynomial ring over a not necessarily commutative domain with not necessarily commuting variables, and if $M$ is any $(R,R)$-bimodule, then the trivial extension of $R$ by $M$ is also a ring with only the two trivial idempotents. This observation shows that the class of $1$-Peirce rings is a large and diverse class.
\end{example}

%Avoiding misunderstanding and unifying similar notions a ring is said to be a \emph{generalized primary ring} if its semisimple factor ring, i.e., the factor by the Jacobson radical, is a division ring %or either an endomorphism ring of a vector space of arbitrary dimension over a division ring or a simple ring.

\begin{definition}
\label{jacobson}  An idempotent $e$ in a ring $R$ is called \emph{{\bf J}-trivial} if both
$eR(1-e)Re$ and $(1-e)ReR(1-e)$ are contained in the Jacobson radical of~$R, \ {\bf J}(R)$. A ring~$R$ is called a~\emph{1-{\bf J} ring} if 0 and 1 are the only {\bf J}-trivial idempotents of~$R$. Inductively, for a natural number $n>1$, a ring $R$ 
is called an \emph{n-{\bf J} ring} if there is a {\bf J}-trivial idempotent $e\in R$ such that $eRe$ is an 
$m$-{\bf J} ring for some $1\leq m<n$ and $(1-e)R(1-e)$ is an $(n-m)$-{\bf J} ring. A {\bf J}-trivial
idempotent $e\in R$ is called an $n$-{\bf J} idempotent if $eRe$ is an $n$-{\bf J} ring. In particular, a 
ring $R$ is called \emph{1-primary} if it is a simple ring or the 
endomorphism ring of an infinite dimensional vector space over a division 
ring, and $R$ is called \emph{n-primary} $(n>1)$ if there is a {\bf J}-trivial idempotent $e\in R$ such that $eRe$ is an $m$-primary ring for some
$1\leq m<n$ and $(1-e)R(1-e)$ is an $(n-m)$-primary ring.
\end{definition}

Since Jacobson-semisimple rings are obviously semiprime, {\bf J}-trivial idempotents in Jacobson-semisimple rings  are precisely Peirce trivial idempotents, whence they are central. Consequently, 
$n$-{\bf J} rings are well-defined in either the class of Jacobson-semisimple rings or in the class of rings where idempotents can be lifted modulo the Jacobson radical, according to Theorem \ref{main}. However, it is possible that there exists a ring which is at the same time both an $m$- and an $n$-{\bf J} ring for different natural numbers $m$ and $n$. Therefore it is an interesting question to determine classes of rings where the notion of $n$-{\bf J} ring is well-defined. 

%\begin{remark}
%\label{warning1} It is worth noting that our notion has no connection to the already existing notion of (commutative) Jacobson rings which are useful in the treatment of the Hilbert Nullstellensatz! The %ring
%$R=\{\frac{m}{n}: \ m,n\in \Z, \ (n,2)=(n,3)=(n,5)=1\}$  is obviously a 1-{\bf J} ring, but its Jacobson semisimple factor ring is  a direct product of the fields $\Z_2, \Z_3$ and $\Z_5$, whence it is %not primitive. 
%\end{remark}

The problem of lifting idempotents modulo the Jacobson radical arises naturally in the investigation of $n$-{\bf J} rings in view of the following calculations. If $e\in R$ is any {\bf J}-trivial idempotent in a ring $R$, and  
$$f=\left[ \begin{array}{cc}
a & b \\
c & d \end{array}\right] \in \left[ \begin{array}{cc}
eRe & eR(1-e)\\
(1-e)Re & (1-e)R(1-e) \end{array}\right]$$
is an arbitrary idempotent in $R$, then the equality
$f^2=\left[ \begin{array}{cc}
a^2+bc & ab+bd \\
ca+dc & d^2+cb \end{array}\right]=f$ implies that $a$ and $d$ are uniquely
determined idempotents modulo the Jacobson radical by $f$. This justifies the following notion:

\begin{definition}
\label{weaklifting} A ring $R$ is called a \emph{weakly lifting ring} if central idempotents (in the semisimple factor by the Jacobson radical) can be lifted (obviously, to  {\bf J}-trivial idempotents.) 
\end{definition}

The next result is obvious.

\begin{lemma}
\label{idempotent} If $R$ is a weakly lifting ring and $e^2=e\in R$ is any {\bf J}-trivial idempotent, then
$eRe$ is also a weakly lifting ring.
\end{lemma}

We are now in a position to generalize  the classical structure theory of semi-perfect rings  as follows.

\begin{corollary}
\label{semiperfect} 
\begin{enumerate} 
\item If $R$ is an $n$-{\bf J} ring, then the Jacobson semisimple factor of~$R$ is a direct sum of $n$ semisimple rings which are, in general, not $1$-Peirce rings.  
\item If $R$ is, in addition, a weakly lifting ring, then all these direct summands  are $1$-Peirce rings.  Furthermore, in this case of a weakly lifting $n$-{\bf J} ring~$R$,  all sets $\{f_1, \ldots, f_m\}$ of pairwise orthogonal {\bf J}-trivial idempotents with sum 1 such that all the subrings $f_iRf_i$ are 1-{\bf J} rings, have $n$ elements, i.e., $m=n$. Moreover these sets are permuted by inner automorphisms.  
\item A ring $R$ is an $n$-primary ring if and only if there are $n$ pairwise orthogonal {\bf J}-trivial idempotents $e_1, \ldots, e_n$, whose sum is 1, such that all the $e_iRe_i, \ i=1, \ldots, n,$ 
are 1-primary rings. Any set $\{f_1, \ldots, f_m\}$ of pairwise orthogonal {\bf J}-trivial idempotents, with sum 1, such that all the $f_iRf_i, \ i=1, \ldots , m,$ are 1-primary rings, has $n$ elements, i.e., $m=n$. Furthermore, if $g^2=g\in R$ is any {\bf J}-trivial idempotent, then there is a uniquely determined natural number $k\leq n$ such that $g$ can be written as a sum of $k$ pairwise orthogonal {\bf J}-trivial idempotents $g_j$ such that all the $g_jRg_j, \ j=1,\ldots , k$, are 1-primary rings. The projective module $_RRg$ is isomorphic to the projective module~$_RRe$, where $e$ is a sum of $k$ appropriate  idempotents~$e_{i_t}, \ t=1,\ldots, k$. 
\item An $n$-primary ring is semiperfect if idempotents can be lifted modulo the Jacobson radical and the semisimple factor is a finite direct sum of matrix rings.
\end{enumerate}
\end{corollary}

The proof of this result can be carried out in the same way as it was carried out in Section 2 for similar results on (strict) $n$-Peirce rings.
Since a semisimple $n$-{\bf J} ring is a direct sum of $n$ semisimple $1$-Peirce rings, one has the following result. 

\begin{theorem}
\label{semisimple} If $R$ is a weakly lifting $n$-{\bf J} ring, then its semisimple factor is a direct sum of $n$ semisimple $1$-Peirce  rings $\bar R_i, \ n$ is an invariant of $R$ and every 
{\bf J}-trivial idempotent $e$ of $R$ maps to the identity of a product of some $\bar R_i$ in the semisimple factor $\bar R$ of $R$.
Conversely, if the semisimple factor of  a ring~$R$ is a direct sum of $n$ semisimple 1-Peirce rings and the corresponding pairwise orthogonal idempotents can be lifted to pairwise orthogonal idempotents, then $R$ is a weakly lifting $n$-{\bf J} ring.
\end{theorem}
Another short way to verify this result is by passing to the semisimple factor which is a direct sum of $n$ semisimple $1$-Peirce rings, then applying the corresponding results on $n$-Peirce rings and thereafter lifting them by using~\cite[Proposition~III.8.1]{j}.

We now list some problems related to this classical topic of lifting idempotents.

\begin{problems}
\label{lifting} Semiperfect rings are characterized as complemented rings. It would be nice to give a constructive proof that (pairwise orthogonal) idempotents of such rings, even of rings
satisfying AB$5^*$ can be lifted to (pairwise orthogonal) idempotents. Recall that a ring {\it satisfies the condition} AB$5^*$ {\it on the right} if the lattice of right ideals is lower continuous, i.e., for any right ideal $K$ and any set~$I_{\a}$ of right ideals downward directed by inclusion one has $K+\cap I_{\a}=\cap (K+I_{\a})$. An open question is whether there are
$n$-primary rings which are not semiperfect, i.e., which are not matrix rings over local rings. It is quite an interesting enterprise to develop the theory of such $n$-{\bf J} rings where idempotents can be lifted. For example, all commutative local rings are semiperfect, because they have only the trivial idempotents 0 and 1, which obviously can be  lifted modulo the Jacobson radical. However, if a ring does not have an identity, then it is not known whether idempotents modulo the Jacobson radical can be lifted, even in the case of left chain rings. Posner \cite{po} discussed an interesting relation between the question of lifting idempotents in left chain rings and the existence of left but not right primitive rings.  
\end{problems}

Since prime rings are $1$-Peirce rings and semiprime $n$-Peirce rings are direct sums of $n$ semiprime $1$-Peirce rings, a semiprime $n$-Peirce ring is called a \emph{semiprime strict} $n$-\emph{Peirce ring} if it is a direct sum of $n$ prime rings. Semiprime $1$-Peirce rings are not necessarily prime as we have seen at the beginning of this section. The case of primitive rings are more doubtful: both left and right primitive rings are both 1-{\bf J} rings and $1$-Peirce rings, but the converse is not true in view of the ring $R=\{\frac{m}{n} : \ m,n\in \Z, (n,2)=(n,3)=(n,5)=1\}$. It is worth noting that both primeness and (left, right) primitiveness are  matrix invariants, i.e., a matrix ring over a prime ring or primitive rings is again prime or primitive, respectively. Since pairwise orthogonal idempotents can be lifted modulo the prime radical, we obviously have  the following result.

\begin{theorem}
\label{prime} If $R$ is a ring such that the factor by the prime radical is a semiprime strict $n$-Peirce ring, then $R$ has $n$ pairwise orthogonal idempotents $e_1, \ldots , e_n$ whose sum is 1 such that all $e_iR(1-e_i)Re_i$ and $(1-e_i)Re_iR(1-e_i)$  are contained in the prime radical and the factor of every $e_iRe_i$ by its prime radical is prime, $i=1, \ldots , n$.
\end{theorem}

Motivated by the theory of semiperfect rings one can ask for some homological characterization of the class of rings described in the above theorem. In particular, one can introduce the following notions:

\begin{definition}
\label{baer} An idempotent $e$ in a ring $R$ is called \emph{{\bf B}-trivial} if $eR(1-e)Re$ and $(1-e)ReR(1-e)$ are contained in the prime radical ${\bf B}(R)$ of $R$. If 0 and 1 are the only 
{\bf  B}-trivial idempotents of $R$, then $R$ is said to be a \emph{1-{\bf B} ring}. Inductively, for a natural number $n>1$, a ring $R$ is called an $n$-\emph{{\bf B} ring} if there is a {\bf B}-trivial idempotent $e\in R$ such that $eRe$ is an $m$-{\bf B} ring for some $1\leq m<n$ and $(1-e)R(1-e)$ is an $(n-m)$-{\bf B} ring. A semiprime $n$-{\bf B} ring is clearly a semiprime $n$-Peirce ring. 
\end{definition}

More generally, one can introduce the notion of trivial idempotents concerning certain radicals similar to ones defined above for the Jacobson and Baer radicals (e.g., various supernilpotent radicals, see \cite{gw}) and then develop a corresponding structure theory. Results on $n$-Peirce rings can be used to determine  properties of rings concerning such radicals by which their factors are semiprime 
$n$-Peirce rings, whence they are direct sums of $n$ semiprime $1$-Pierce rings together with additional assumptions on lifting idempotents.
%This is the case of \emph{semiprime radicals}, i.e., of such radicals whose factor rings are semiprime, or of any nilradical. 
%For example, this idea is applied to the Brown-McCoy radical which is a semiprime radical but not a nil radical. 
For example, in the case of the Brown-McCoy radical we are interested in finite direct sums of simple rings. Therefore assuming that central idempotents modulo the Brown-McCoy radical can be lifted and the factor ring is a direct sum of finitely many simple rings of a certain kind, one can develop a structure theory based on Brown-McCoy trivial idempotents.

If $A=\Z_4$ and $R=\left[ \begin{array}{cc}
A & A \\
2A & A \end{array}\right]$, then $R$ is a $1$-Peirce ring but not a $1$-{\bf B} ring, because $e=\left[ \begin{array}{cc}
1 & 0 \\
0 & 0 \end{array}\right]$ is a {\bf B}-trivial idempotent of $R$.

Note that for radicals $\rho$  such that $\rho(R)\subseteq {\bf J}(R)$, a nonzero idempotent in $R$ remains nonzero in $R/\rho(R)$.  This is not so for radicals not contained in ${\bf J}(R)$.  For example, the Brown-McCoy radical, ${\bf G}(R)$, may contain nontrivial idempotents.  However, any nonzero inner Peirce trivial idempotent is not an element of ${\bf G}(R)$. 
 
Since finitely many pairwise orthogonal idempotents can be lifted modulo the prime radical, we have the following generalization of Theorem \ref{prime}.

\begin{theorem}
\label{semiprime} $R$ is an $n$-{\bf B} ring if and only if its factor by the prime radical is a direct sum of $n$
semiprime 1-Peirce rings. In particular, $n$ is an invariant of~$R$, i.e., there are  $n$  pairwise orthogonal {\bf B}-trivial idempotents $e_i, \  i=1, \ldots, n,$ in $R$, with sum~1, such that all
$e_iRe_i, \ i=1, \ldots, n,$ are 1-{\bf B} rings. Moreover,  every {\bf B}-trivial idempotent $f\in R$ can be written as
a sum of $m$ pairwise orthogonal {\bf B}-trivial idempotents $f_i, \ i=1,\ldots, m, \ m\leq n$, such that all the $f_iRf_i$ 
are 1-{\bf B} rings (hence semiprime indecomposable rings), and there is an idempotent $e\in R$ which is a sum of $m$ appropriate idempotents 
$e_i, \ i=1, \ldots, n,$ such that $_RRe$ and $_RRf$ are isomorphic. In particular, if $\{f_1,\ldots, f_m\}$ is an arbitrary set of pairwise orthogonal {\bf B}-trivial idempotents with sum 1 and all 
the $f_iRf_i$ are 1-{\bf B} rings, then $m=n$ and there is a permutation $\sigma$ of $\{1,2,\ldots,n\}$ and an invertible element $u\in R$ such that 
$e_i=uf_{\sigma(i)}u^{-1}$ for all $i=1, \ldots, n$.
\end{theorem}

%An immediate application is the following description of  quasi-Baer $n$-{\bf B} rings,  which is similar to the description of hereditary noetherian prime rings. Recall that a 
% ring $R$ is called a \emph{quasi-Baer ring} if the right annihilator of every ideal of~$R$ is an idempotent generated right ideal of $R$.

By Theorem \ref{semiprime} it is an interesting problem to search for good homological characterizations of classes of rings described in Theorem \ref{semiprime}, even when it is assumed that the semiprime factor ring is a direct sum of the corresponding $n$ prime rings. It would also be interesting to compare the classes of prime rings and semiprime
1-Peirce rings.

%\begin{corollary}
%\label{Baer} Let $R$ be a $n$-Peirce ring.  If $R$ satisfies any of the following conditions, then $R$ is isomorphic to an $n$-by-$n$ generalized matrix ring with prime rings on %the diagonal which satisfy the same corresponding condition:
%\begin{enumerate}
%  \item quasi-Baer and right or left (semi-)hereditary;
%  \item  quasi-Baer and right or left perfect; or
%  \item  regular right or left selfinjective.
%\end{enumerate}
%\end{corollary} 

%\begin{proof} This result follows from Theorem \ref{quasibaer}, \cite[Lemma 5.13]{abivw1}, the fact that if~$R$ satisfies any of conditions (1)-(3) then $eRe$ satisfies the same %corresponding condition for any 
%$e = e^2 \in R$, and that condition (3) implies that $R$ is a Baer (hence quasi-Baer) ring.
%\end{proof}

\section{Applications}

In this section we apply our theory of $n$-Peirce rings to various important classes of rings. Observe that from our previous results each $n$-Peirce ring is isomorphic to a generalized matrix ring
$$R' = \left[ \begin{array}{cccc}
R_1 & M_{12} & \cdots & M_{1n}\\
M_{21} & R_{2} & \ddots & \vdots \\
\vdots & \ddots & \ddots & M_{n-1,n} \\
M_{n1} & \cdots & M_{n,n-1} & R_n \end{array}\right],$$

\noindent where each $R_i$ is a $1$-Peirce ring, each $M_{ij}$ is an $(R_i, R_j)$-bimodule and $M_{ij}M_{ji}=0_{R_i}$ for all $i\ne j$, and $R_1,\ldots,R_n$ are unique up to isomorphism and permutation.
The fact that $M_{ij}M_{ji}=0_{R_i}$ simplifies the matrix multiplication of elements of $R'$. For example, it is relatively easy to compute idempotents. Also the calculation of various radicals which contain the prime radical (e.g., the Jacobson, nil, and Brown-McCoy radicals) is reduced to computing the radicals of the rings $R_i$, since $\mathcal{D}(R)^-$ is nilpotent.

Since the notion of a (quasi-)Baer ring will play a role in the main results of this section, recall: a ring $R$ is {\it (quasi-)Baer} if for each nonempty $X\subseteq R$ ($X$ an ideal of $R$) there is 
an $e=e^2\in R$ such that $\ul{r}(X)=eR$, where $\ul{r}(X)$ denotes the right annihilator of $X$ in $R.$ The class of quasi-Baer rings is ubiquitous, since it contains all: Baer rings (hence endomorphism rings of vector spaces over division rings), AW$^*$-algebras (in particular, von Neumann algebras), regular right selfinjective rings, prime rings, and biregular right selfinjective rings. Moreover, the class of quasi-Baer rings is closed under direct products, matrix rings, triangular matrix rings, and various polynomial extensions. Furthermore, each semiprime ring has a quasi-Baer hull contained in 
its (Martindale) symmetric ring of quotients; for more details, see \cite{bipar}.

\begin{lemma}
\label{innertrivial} 
\begin{enumerate}
\item \cite[Lemma 3.4]{abivw1} Let $R$ be a ring and $e\in R$. Then $e$ is an inner Peirce trivial idempotent if and only if $h: R \rightarrow eRe$, defined by $h(x) = exe$, is a surjective ring homomorphism.
\item \cite[Lemma 5.13]{abivw1} $R$ is a prime ring if and only if $R$ is a quasi-Baer $1$-Peirce ring.
\end{enumerate}
\end{lemma}

\begin{lemma}
\label{primitive} Let $R$ be a ring and $0\ne c$ a Peirce trivial idempotent of $R$.
\begin{enumerate}
\item Let $f=f^2\in R$ be primitive and $c_1$ a Peirce trivial idempotent of $cRc$. Then:
\begin{enumerate}
\item $fcf=f=fc_1f$; and
\item $cfc$ is a primitive idempotent of $R$.
\end{enumerate} 
\item Let $\{f_1,\dots,f_k\}$ be a complete set of primitive orthogonal idempotents of~$R$. Then there exists $H\subseteq \{1,\ldots,k\}$ such that $\{cf_hc \ \vert \ h\in H\}$ is a set of primitive orthogonal idempotents of $R$ such that $c=\sum\limits_{h\in H} cf_hc$.
\end{enumerate}
\end{lemma}

\begin{proof}
(1) By \cite[Proposition 5.4(1)]{abivw1}, (a) holds. Since $c$ is inner Peirce trvial, $cfc=(cfc)^2$. To show that $cfc$ is primitive, we prove that $cfc$ is the only nonzero idempotent in $cfcRcfc$. Let $0\ne cxc = (cxc)^2\in cfcRcfc$, where $x=fcycf$ for some $y\in R$. Observe that $cxc=cxccxc=cx^2c$, since $c$ is inner Peirce trivial. Consider 

$\begin{array}{lll}  (fcxcf)^2 & = & fcxcffcxcf \\ 
& = & fcx(cfc)xcf \\ 
& = & fcxcxcf \\
& = & fcx^2cf \\
& = & fcxcf\in fRf. \end{array}$

\noindent Observe that $c(fcxcf)c=cfc(cxc)cfc=cxc\ne 0.$ Hence $fcxcf\ne 0.$ Since $f$ is primitive, $f=fcxcf.$ Then

$\begin{array}{lll}  cfc & = & c(fcxcf)c \\ 
& = & (cfc)(cxc)(cfc) \\ 
& = & cxc. \end{array}$

\noindent Therefore $cfc$ is primitive, so (b) holds.

(2) Take $H= \{h\in \{1,\ldots,k\} \ \vert \ cf_hc\ne 0\}.$ Let $h,j\in H$ such that $h\ne j.$ Then $(cf_hc)(cf_jc)=cf_hf_jc=0$ because $c$ is inner Peirce trivial. Using (1) we now have that (2) holds.
\end{proof}

In the next result, see \cite{mcro} for details on Krull dimension.

\begin{theorem}
\label{dcc} 
If $R$ satisfies any of the following conditions, then $R$ is an $n$-Peirce ring with a complete set of orthogonal idempotents, $\{e_1,\ldots,e_n\}$, and a generalized matrix representation 
$$R\cong \left[ \begin{array}{cccc}
R_1 & M_{12} & \cdots & M_{1n}\\
M_{21} & R_{2} & \ddots & \vdots \\
\vdots & \ddots & \ddots & M_{n-1,n} \\
M_{n1} & \cdots & M_{n,n-1} & R_n \end{array}\right],$$

\noindent where each $R_i=e_iRe_i$ is a $1$-Peirce ring satisfying the same condition as $R$, each $M_{ij}=e_iRe_j$ with $M_{ij}M_{ji}=0_{R_i}$ for all $i\ne j$, and $R_1,\ldots,R_n$ are unique up to isomorphism and permutation.
\begin{enumerate}
\item $R$ has DCC on $\{ReR \ \vert \ e=e^2\in R \ \hbox{is Peirce trivial}\}$.
\item $R$ has a complete set of primitive orthogonal idempotents.
\item $R$ has no infinite set of orthogonal idempotents.
\item $R_R$ has Krull dimension.
\item $R$ is semilocal.
\item $R$ is semiperfect.
\item $R$ is left (or right) perfect.
\item $R$ is semiprimary.
\end{enumerate}
\end{theorem}

\begin{proof}
From \cite[Theorem 5.7(1)]{abivw1} and Theorem \ref{peirce4}, $R$ is an $n$-Peirce ring with a complete set of orthogonal idempotents, $\{e_1,\ldots,e_n\}$, and the indicated generalized matrix representation, 
where each $R_i=e_iRe_i$ is a $1$-Peirce ring, each $M_{ij}=e_iRe_j$ with $M_{ij}M_{ji}=0_{R_i}$ for all $i\ne j$, and $R_1,\ldots,R_n$ are unique up to isomorphism and permutation. It only remains to show that if $R$ satisfies any of the conditions (1) - (8) then so does each $R_i$.

For condition (1), the result follows from \cite[Theorem 5.11]{abivw1}. So assume condition~(2) that $R$ has  a complete set of primitive orthogonal idempotents, $\{f_1,\ldots,f_k\}$. Using Corollary
\ref{structure} and Lemma \ref{primitive}, then $e_if_je_i=0$ or $e_if_je_i$ is primitive for each $i=1,\ldots,n$ and $j=1,\ldots,k$. Then for each $R_i$ there exists $H_i\subseteq \{1,\ldots,k\}$ such that $\{e_if_he_i \ \vert \ h\in H\}$ is a complete set of primitive orthogonal idempotents for~$R_i.$

If $R$ satisfies any of conditions (3) - (8), then each $R_i$ contains no infinite set of orthogonal idempotents and satisfies the same condition as $R$.
\end{proof}

\begin{corollary}
\label{perfect}
Assume $R$ is left perfect in Theorem \ref{dcc}. Then each $R_i$ is either simple Artinian or $\left[\hbox{Soc}({R_i}_{R_i})\right]^2=0.$ If $R$ is also quasi-Baer, then each $R_i$ is simple Artinian and $R$ is semiprimary.
\end{corollary}

\begin{proof}
Observe that a $1$-Peirce ring is semicentral reduced. Now from \cite[Theorem~3.13]{bikipa} each $R_i$ is either simple Artinian or $\left[\hbox{Soc}({R_i}_{R_i})\right]^2=0.$ The remainder of the proof follows from Lemma \ref{innertrivial}(2). Since ${\bf J}(R)=\mathcal{D}(R)^-$, \cite[Proposition 4.4]{abivw1} and Corollary \ref{structure1} yield that $R$ is semiprimary (also see \cite[Theorem 2.3]{bikipa2}).
\end{proof}

The following examples illustrate Corollary \ref{perfect}.

\begin{examples} 
\begin{enumerate}
\item Let $F$ be a field, $S$ the ring of $k$-by-$k$ upper triangular matrices over $F$, and $R$ is the $n$-by-$n$ upper triangular matrix ring over~$S$, where $k,n\ge 1$. Then $R$ is a semiprimary quasi-Baer $kn$-Peirce ring which is not a Baer ring, where each $R_i$ is isomorphic to $F$ (see \cite{polza}). 
\item Let $A$ be an Artinian $1$-Peirce ring such that $A/J$ is a simple ring, $J^3=0$, and $J^2\ne0$, where $J$ is the Jacobson radical of $A$ (e.g., $A=\Z/8\Z$). Let 
$R=\left[ \begin{array}{ccc}
A & A/J \\
0 & A/J \end{array}\right]$. Then $R$ is a $2$-Peirce ring with $R_1=A$ and $\left[\hbox{Soc}({R_1}_{R_1})\right]^2=0$, and $R_2$ is a simple Artinian ring.
\item Let $A$ be as in (2). Let 
$$R=\left[ \begin{array}{ccc}
A & J^2 & J \\
J^2 & A & J^2\\
J^2 & J & A \end{array}\right].$$ Then $R$ is a $3$-Peirce ring where each $R_i=A$, hence $\left[\hbox{Soc}({R_i}_{R_i})\right]^2=0$ for all~$i$.
\end{enumerate}
\end{examples}

\begin{theorem}
\label{quasibaer} If $R$ is an $n$-{\bf B} ring such that the semiprime factor by the prime radical is quasi-Baer, then $R$ is an $n\times n$ generalized matrix ring with rings having prime factors by the prime radical on the diagonal. 
\end{theorem}

\begin{proof} The condition that the factor by the prime radical is quasi-Baer implies that this factor is a direct sum of prime rings. Therefore the result follows immediately from Theorem \ref{prime}. 
\end{proof}

\begin{proposition}
\label{complete}
Let $R$ be an $n$-Peirce ring. Then $R$ has a complete set of orthogonal idempotents, $\{e_1,\ldots,e_n\}$, and a generalized matrix representation 
$$R\cong \left[ \begin{array}{cccc}
R_1 & M_{12} & \cdots & M_{1n}\\
M_{21} & R_{2} & \ddots & \vdots \\
\vdots & \ddots & \ddots & M_{n-1,n} \\
M_{n1} & \cdots & M_{n,n-1} & R_n \end{array}\right],$$

\noindent where each $R_i=e_iRe_i$ is a $1$-Peirce ring, each $M_{ij}=e_iRe_j$ with $M_{ij}M_{ji}=0_{R_i}$ for all $i\ne j$, and $R_1,\ldots,R_n$ are unique up to isomorphism and permutation. Moreover, if $R$ satisfies any condition which transfers from $R$ to a homomorphic image or to $eRe$, where $e=e^2\in R$, then each $R_i$ also satisfies the condition.
\end{proposition}

\begin{proof}
This result is a consequence of \cite[Theorem 5.7(1)]{abivw1}, Theorem \ref{peirce4}(1) and Lemma \ref{innertrivial}.
\end{proof}

To indicate the applicability of Proposition \ref{complete}, the following is a list of some of the classes of rings which are closed with respect to homomorphic images or contain $eRe$ 
whenever $e=e^2\in R$ and $R$ is in the following class: right Noetherian, right (semi-)Artinian, PI (i.e., satisfies a polynomial identity), (quasi-)Baer, right (semi-)hereditary, (bi-, $\pi$-, semi-)regular, $I$-ring (i.e., every non-nil right ideal contains a nonzero idempotent), bounded index of nilpotency, right selfinjective, etc. Furthermore, in Proposition \ref{complete}, if $R$ satisfies any of the above conditions and is quasi-Baer, then each $R_i$ is a prime ring satisfying the condition.

We conclude the paper with the following question:
\medskip

\begin{question} If R is an $n$-Pierce, under what conditions is the right classical ring of quotients or the maximal right ring of quotients also a $k$-Peirce ring for some positive integer $k$?
\end{question}

\end{document}